\documentclass[12pt,twoside]{article}
\usepackage{amsmath, amsthm, amssymb}
\usepackage[colorlinks]{hyperref}
\usepackage{color}
\usepackage{enumerate}

\numberwithin{equation}{section}
\theoremstyle{plain}
\swapnumbers

\newtheorem{theorem}{\hskip-1mm.\,Theorem}[section]
\newtheorem{lemma}[theorem]{\hskip-1mm.\,Lemma}
\newtheorem{proposition}[theorem]{\hskip-1mm.\,Proposition}
\newtheorem{corollary}[theorem]{\hskip-1mm.\,Corollary}
\theoremstyle{definition}
\newtheorem{definition}[theorem]{\hskip-1mm.\,Definition}
\newtheorem{remark}[theorem]{\hskip-1mm.\,Remark}

\DeclareMathOperator{\var}{var}

\begin{document}

\newcommand{\dd}{\hskip0.2mm\mbox{\rm d}}
\newcommand{\N}{{\mathbb N}}
\newcommand{\R}{{\mathbb R}}
\newcommand{\dis}{\displaystyle}
\newcommand{\ab}{[a,b\,]}
\newcommand{\ia}{\noalign{\noindent{and}}}
\newcommand{\eps}{\varepsilon}
\newcommand{\inn}{\,{\in}\,}
\def\skipaline{\removelastskip\vskip10pt plus 1pt minus 1pt}
\def\skiphalfaline{\removelastskip\vskip5pt plus 1pt minus 1pt}

\title{Bounded convergence theorem for abstract Kurzweil-Stieltjes integral}

\author{
Giselle Antunes Monteiro
\hskip-1mm\thanks{Mathematical Institute, Academy of Sciences of Czech Republic, Prague, Czech Republic
(email: gam@math.cas.cz). Supported by RVO: 67985840 and by the Academic Human Resource Program
of the Academy of Sciences of the Czech Republic.}\,\,,
Umi Mahnuna Hanung
\hskip-1mm\thanks{Department of Mathematics, Universitas Gadjah Mada, Sekip Utara Bulaksumur,
55281 Yogyakarta, Indonesia;
FNWI, Korteweg de Vries Instituut voor Wiskunde, Universiteit van Amsterdam, P.O. Box 94248,
1090 GE Amsterdam, the Netherlands;
Mathematical Institute, Academy of Sciences of Czech Republic, Prague, Czech Republic
(email:  hanungum@ugm.ac.id).
Supported by the grant No.~573/E4.4/K/2011 of the Directorate General of Higher Education (DGHE),
Ministry of Education and Culture, Republic of Indonesia.}\,\,,
and Milan Tvrd\'y
\hskip-1mm\thanks{Mathematical Institute, Academy of Sciences of Czech Republic, Prague, Czech Republic
(email: tvrdy@math.cas.cz). Supported by RVO: 67985840 and by the grant No.~14-06958S of the Grant
Agency of the Czech Republic.}}
\date{\today}
\maketitle

\begin{abstract} 
In the theories of Lebesgue integration and of ordinary differential equations, the Lebesgue
Dominated Convergence Theorem provides one of the most widely used tools. Available analogy
in the Riemann or Riemann-Stieltjes integration is the Bounded Convergence Theorem, sometimes
called also the Arzel\`a or Arzel\`a-Osgood or Osgood Theorem. In the setting of the Kurzweil-Stieltjes
integral for real valued functions its proof can be obtained by a slight modification of the proof
given for the $\sigma$-Young-Stieltjes integral by T.H.~Hildebrandt in his monograph from 1963.
However, it is clear that the Hildebrandt's proof cannot be extended to the case of Banach space-valued
functions. Moreover, it essentially utilizes the Arzel\`a Lemma which does not fit too much into
elementary text-books. In this paper, we present the proof of the Bounded Convergence Theorem for
the abstract Kurzweil-Stieltjes integral in a setting elementary as much as possible.
\end{abstract}

\smallskip

\noindent{2010 {\it Mathematics Subject Classification}: 26A39, 28B05

\smallskip


\smallskip

\section{Introduction}

In the theories of Lebesgue integration and of ordinary differential equations, the Lebesgue
Dominated Convergence Theorem provides one of the most widely used tools. Available analogy
in the Riemann or Riemann-Stieltjes integration is the {\em Bounded Convergence Theorem},
sometimes called also the {\em Arzel\`a} or {\em Arzel\`a-Osgood}  or {\em Osgood Theorem}.
In the setting of the Kurzweil-Stieltjes integral for real valued functions this result reads
as follows:

{\em If $F\,{:}\,\ab\to\R$ has a bounded variation on $\ab,$  $g\,{:}\ab\to\R$ is
regulated on $\ab$ and the sequence $\{g_n\}$ of functions regulated on $\ab$
is such that
\[
   \lim_{n\to\infty}g_n(t)=g(t)\quad \mbox{for \ } t\in\ab
\]
and
\[
   \|g_n\|_\infty\le K<\infty \quad\mbox{for \ } n\in\N,
\]
then}
\[
   \lim_{n\to\infty}\int_a^b\dd[F]\,g_n=\int_a^b\dd[F]\,g.
\]
The proof can be easily obtained by a slight modification of the proof given
for the $\sigma$-Young integral by Hildebrandt in \cite[Theorem II.19.3.14]{Hi},
cf.~\cite[Theorem I.4.24]{STV}. However, it is clear that the proof by Hildebrandt
cannot be extended to the case of Banach space-valued functions. Moreover, it
essentially utilizes the Arzel\`a Lemma which does not fit too much into elementary
text-books. Even in the case of the Riemann integral the elementary proof of
the corresponding Bounded Convergence Theorem has been for a long time considered
to be rather impossible as stated by Lewin in \cite{Lewin}:
\begin{quote}
``The bounded convergence theorem follows trivially from the Lebesgue dominated
convergence theorem, but at the level of an introductory course in analysis, when
the Riemann integral is being studied, how hard is the bounded convergence theorem?
For an answer, we might look at Bartle and Sherbert \cite{BaSh}, page 203:
``The proof of this result is quite delicate and will be omitted".
Or we might look at Apostol \cite{Apo}, page 228:
``The proof of Arzela's theorem is considerably more difficult than ... and will not
be given here".
Walter Rudin in \cite{Ru1} ignores the theorem altogether in his chapter on Riemann
integration, presenting it only as a corollary to the Lebesgue dominated
convergence theorem several chapters later, and in \cite{Ru2}, in an interesting
problem in Chapter Two, Rudin refers his readers to \cite{E}. In \cite{E},
Eberlein does present a proof which from some points of view is elementary. Certainly,
his proof does not require any notions of measurability, but it is hardly elementary
from the point of view of a student who is first learning the Riemann integral."
\end{quote}
Despite that, in earlier times, many other authors have dedicated themselves to obtaining
a proof independent of the theory of Lebesgue measure for such convergence result.
See for instance \cite{Lu}
and the references therein. Of course, in \cite{Lewin}, Lewin succeeded in his search
for an elementary proof that, as he claimed, ``could be included for the first time in
an introductory course''.

\smallskip

In this paper, we utilize some of the Lewin's ideas and present the proof of the Bounded
Convergence Theorem for the abstract Kurzweil-Stieltjes integral in a setting elementary
as much as possible.

\section{Preliminaries}
Throughout this paper $X$ is a Banach space and $L(X)$ is the Banach space of all bounded
linear operators on $X.$ By $\|\cdot\|_X$ we denote the norm in $X,$ while
$\|\cdot\|_{L(X)}$ denotes the usual operator norm in $L(X).$

\skipaline

For $-\infty\,{<}\,a\,{<}\,b\,{<}\,\infty,$ $\ab$ and $(a,b)$ denote the corresponding
closed and open intervals, respectively. Furthermore, $[a,b)$ and $(a,b]$ are
the corresponding half-open intervals while $[c]$ denotes the degenerate interval consisting
of a single real number $c\in\ab$.

\skipaline

A set $D=\{\alpha_0,\alpha_1,\dots,\alpha_{\nu(D)}\}\subset\ab$ with $\nu(D)\in\N$
is said to be a division of $\ab$ if
\[
    a\,{=}\,\alpha_0\,{<}\,\alpha_1\,{<}\,\dots\,{<}\,\alpha_{\nu(D)}\,{=}\,b\,.
\]
The set of all divisions of $\ab$ is denoted by $\mathcal D\ab.$ The symbol $\nu(D)$
will be kept for the number of subintervals $[\alpha_{j-1},\alpha_j]$ generated by
the division $D.$

\skipaline

A function $f:\ab\to X$ is called a finite step function on $\ab$ if there exists
a division $D=\{\alpha_0,\alpha_1,\dots,\alpha_m\}$ of $\ab$ such that $f$ is constant
on every open interval $(\alpha_{j-1},\alpha_j),\,j=1,2,\dots,m.$

\skipaline

For an arbitrary function $f{:}\,\ab\,{\to}\,X$ we set
\[
    \|f\|_{\infty}=\sup_{t\in\ab}\|f(t)\|_X
\]
and the variation of $f$ on $\ab$ is given by
\[
   \var_a^b f
   =\sup_{D\in\mathcal D\ab}\sum_{j=1}^{\nu(D)}\|f(\alpha_j)-f(\alpha_{j-1})\|_X.
\]
Sometimes, $\var_a^b f$ is also called the {\em Jordan variation} of $f$ on $\ab.$
If $\var_a^b f<\infty$ we say that $f$ is of bounded variation on $\ab.$ $BV(\ab,X)$
denotes the set of all functions $f:\ab\to X$ of bounded variation on $\ab.$

\skipaline

$C(\ab, X)$ is the set of all $X$-valued functions which are continuous on $\ab,$ while
$G(\ab,X)$ denotes the set of all regulated functions $f:\ab\to X.$ Recall that
$f{:}\,\ab\,{\to}\,X$ is regulated on $\ab$ if for each $t\in[a,b)$ there is $f(t+)\in X$
such that
\[
   \lim_{s\to t+}\|f(s)-f(t+)\|_X=0
\]
and for each $t\in (a,b\,]$ there is $f(t-)\in X$ such that
\[
   \lim_{s\to t-}\|f(s)- f(t-)\|_X=0\,.
\]
Furthermore, for $t\in\ab$ we put $\Delta^+f(t)=f(t+)-f(t),$ $\Delta^-f(t)=f(t)-f(t-)$ and
$\Delta f(t)=f(t+)-f(t-)$ (where by convention $\Delta^-f(a)=\Delta^+f(b)=0$).

\skipaline

Given $F\,{:}\,\ab\to L(X),$ we say $F$ is simply-regulated on $\ab$ if, for each $x\in X,$
the function $t\in\ab\to F(t)\,x\in X$ is regulated. In this case, we write
$F\,{\in}\,(\mathcal{B})G(\ab,L(X)).$

\skipaline

Clearly,
\[
   C(\ab,L(X))\subset G(\ab,L(X))\subset(\mathcal{B})G(\ab,L(X)).
\]
Moreover, it is known that $BV(\ab,L(X))\subset G(\ab,L(X)),$ as well.

\skipaline

Now, let us recall the definition of the abstract Kurzweil-Stieltjes integral as introduced
by \v{S}.~Schwabik in \cite{Sch1}.

\smallskip

Systems $P\,{=}\,\{(\tau_j,[\alpha_{j-1},\alpha_j])\,{:}\,j=1,\dots,m\},$ where
$\{\alpha_0,\alpha_1,\dots,\alpha_m\}\in\mathcal{D}\ab$ and
$\tau_j\,{\in}\,[\alpha_{j{-}1},\alpha_j]$ for $j\,{=}\,1,\dots,m,$ are called
{\em tagged divisions} of $\ab.$

\smallskip

Furthermore, functions $\delta\,{:}\,\ab\,{\to}\,(0,\infty)$ are said to be {\em gauges}
on $\ab.$ Given a~gauge $\delta$ on $\ab,$ the tagged division
$P\,{=}\,\{(\tau_j,[\alpha_{j-1},\alpha_j]){:}\,j=1,\dots,m\}$ of $\ab$ is said to be
{\em $\delta$-fine} if
\[
   [\alpha_{j-1},\alpha_j]\subset(\tau_j-\delta(\tau_j),\tau_j+\delta(\tau_j))
   \quad\mbox{for \ } j=1,\dots,m\,.
\]
We remark that for an arbitrary gauge $\delta$ on $\ab$ there always exists a $\delta$-fine
tagged division of $\ab.$ This is stated by the Cousin lemma (see \cite[Lemma 1.4]{Sch}).

\smallskip

For given functions $F\,{:}\,\ab\to L(X)$ and $g\,{:}\,\ab\to X$ and a tagged division
$P\,{=}\,\{(\tau_j,[\alpha_{j-1},\alpha_j])\,{:}\,j=1,\dots,m\}$ of $\ab,$ we define
\begin{align*}
   &S(F,\dd g,P)=\sum_{j=1}^m F(\tau_j)\,[g(\alpha_j)-g(\alpha_{j-1})]
  \\\noalign{\noindent\mbox{and}}
   &S(\dd F,g,P)=\sum_{j=1}^m[F(\alpha_j)-F(\alpha_{j-1})]\,g(\tau_j)\,.
\end{align*}
We say that $I\inn X$ is the Kurzweil-Stieltjes integral (or shortly KS-integral)
of $F$ with respect to $g$ from $a$ to $b$ and write
\[
   I=\int_a^bF\,\dd\,[g]
\]
if for every $\eps>0$ there exists a gauge $\delta$ on $\ab$ such that
\[
   \Big\|S(F,\dd g,P)-I\Big\|_X<\eps
   \quad\mbox{for all \ } \delta\mbox{-fine tagged divisions \ } P \mbox{\ of \ }\ab\,.
\]
Similarly, $J\inn X$ is the KS-integral of $g$ with respect to $F$ from $a$ to $b$
if for every $\eps>0$ there exists a gauge $\delta$ on $\ab$ such that
\[
   \Big\|S(\dd F,g,P)-J\Big\|_X<\eps
   \quad\mbox{for all \ } \delta\mbox{-fine tagged divisions \ } P \mbox{\ of \ }\ab\,.
\]
In this case we write $\dis J=\int_a^b\dd\,[F]\,g.$

For the basic properties of the abstract KS-integral, we refer to \cite{Sch1}--\cite{Sch3}
and \cite{MT1}--\cite{MT2}.

\section{Jordan decomposition}\label{sec3}
In this section we will show that, as in the case of real valued functions, any
function of bounded variation on $\ab$ can be written as the sum of a continuous function and a break
function.

Recall that a function $f{:}\,\ab\,{\to}\,X$ is said to be a break function if there exist sequences
\[
   \{s_k\}\subset\ab,\quad \{c_k\}\subset X,\quad \{d_k\}\subset X,
\]
such that
\begin{align}\label{break-1}
   &s_k\ne s_\ell \quad\mbox{if \ } k\ne\ell,\quad
   \sum_{k=1}^\infty\big(\|c_k\|_X+\|d_k\|_X\big)<\infty,
  \\\ia\label{break-2}
   &f(t)=\dis\sum_{k=1}^\infty\Big(c_k\,\chi_{(s_k,b]}(t)+d_k\,\chi_{[s_k,b]}(t)\Big)
   \quad\mbox{ \ for \ }t\in\ab.
\end{align}
Since for any $t\in\ab,$ the series on the right-hand side of \eqref{break-2} converges
absolutely, it converges also unconditionally (see \cite[Chapter VI]{Di}). Hence we can
write also
\begin{equation}\label{break-3}
  f(t)=\sum_{a\le s_k<t}c_k+\sum_{a<s_k\le t}d_k \quad\mbox{for \ }t\in\ab.
\end{equation}

\skipaline

It is not difficult to see that, if $f$ is given by \eqref{break-3} (or \eqref{break-2}),
then
\[
    \Delta^+f(s_k)=c_k\quad\mbox{and}\quad
    \Delta^-f(s_k)=d_k\quad\mbox{for \ }k\in\N.
\]
In particular, the one-sided limits are given by the expressions
\begin{align*}
   f(t+)&=\sum_{a\le s_k\le t}c_k+\sum_{a<s_k\le t}d_k \quad\mbox{ \ for \ }t\in[a,b),
  \\
   f(t-)&=\sum_{a\leq s_k<t}c_k+\sum_{a<s_k<t}d_k \quad\mbox{ \ for \ }t\in(a,b].
\end{align*}
Moreover, $f\in BV(\ab,X)$ and
\[
   \var_a^bf=\sum_{t\in[a,b)}\|\Delta^+f(t)\|_X+\sum_{t\in(a,b]}\|\Delta^-f(t)\|_X.
\]

\smallskip

\begin{theorem}\label{jordan}
If $f\in BV(\ab,X)$ then there exists a break function $f^B{:}\,\ab\,{\to}\,X$ such that
$f-f^B$ is continuous on $\ab.$ Moreover, the function $f^B$ is determined uniquely up
to an~additive constant.
\end{theorem}
\begin{proof} Recall that a regulated function can have at most a countable number
of points of discontinuity (see \cite[Corollary I.3.2]{H}). Let $\{s_k\}$ be the set
of points of discontinuity of $f$ in $\ab$ and consider the function
$f^B{:}\,\ab\,{\to}\,X$ given by
\begin{equation}\label{f-2}
    f^B(t)=\sum_{a\le s_k<t}\Delta^+f(s_k)+\sum_{a<s_k\le t}\Delta^-f(s_k)
    \quad\mbox{for \ } t\in\ab.
\end{equation}
By \cite[Lemma 4.1]{MT2}, both series on the right-hand side of \eqref{f-2} are absolutely
convergent, thus $f^B$ is a break function. Moreover, for $t\in\ab$ we have
\[
   \Delta^+f^B(t)=\Delta^+f(t)
   \quad\mbox{and}\quad
   \Delta^-f^B(t)=\Delta^-f(t).
\]
This implies that the function $f^{\,C}=f-f^B$ is continuous on $\ab.$

\smallskip

To prove uniqueness, assume that $f=\widetilde{f}^{\,C}+\widetilde{f}^B,$ where
$\widetilde{f}^{\,C}$ is continuous and $\widetilde{f}^B$ is a break function. Let
\[
   \widetilde{f}^B(t)=\sum_{a\le t_k<t}c_k+\sum_{a<t_k\le t}d_k \quad\mbox{for \ } t\in\ab,
\]
for some sequences $\{t_k\}\subset\ab,$ \ $\{c_k\},\,\{d_k\}\subset X$ fulfilling
\eqref{break-1}. For each $k\in\N,$ one can easily verify that
\begin{equation}\label{c_k}
   c_k=\Delta^+\widetilde{f}^{\,B}(t_k)=\Delta^+f(t_k)=\Delta^+f^B(t_k).
\end{equation}
However, this has a sense only if $t_k=s_{\ell_k}$ for some $\ell_k\in\N.$ Analogously,
we can see that
\begin{equation}\label{d_k}
   d_k=\Delta^-\widetilde{f}^{\,B}(t_k)=\Delta^-f^{\,B}(s_{n_k})
\end{equation}
for some $n_k\in\N.$ In view of this and since the series above are unconditionally
convergent (see \cite[Chapter VI]{Di}), we have
$f^B(t)-\widetilde{f}^{\,B}(t)=f^B(a)-\widetilde{f}^{\,B}(a)$ for $t\in\ab.$
This completes the proof.
\end{proof}

\skipaline

\begin{remark}\label{jordan-rem}
Let $f\in BV(\ab,X)$ and let $f^B$ be given as in Theorem~\ref{jordan}. Let us denote $f^{\,C}:=f-f^B.$
Then $f=f^{\,C}+f^B$ and we say that $f$ is decomposed into the sum of its {\em continuous part} $f^{\,C}$
and its {\em break part} $f^B.$ Such a decomposition is in the classical case of real valued functions
called the {\em Jordan decomposition} of $f.$

Furthermore, note that, in view of \eqref{c_k} and \eqref{d_k}, we have also
\[
  f^B(t)=\sum_{k=1}^\infty
         \big[\Delta^+f(s_k)\,\chi_{(s_k,b]}(t)+\Delta^-f(s_k)\,\chi_{[s_k,b]}(t)\big]
  \quad\mbox{for \ } t\in\ab.
\]
Let us put
\[
  f_n^B(t)=\sum_{k=1}^n
         \big[\Delta^+f(s_k)\,\chi_{(s_k,b]}(t)+\Delta^-f(s_k)\,\chi_{[s_k,b]}(t)\big]
  \quad\mbox{for \ } t\in\ab \mbox{ \ and \ } n\in\N.
\]
Then it is not difficult to prove that $\dis\lim_{n\to\infty}\var_a^b(f_n^B-f^B)=0$
(cf.~\cite[Lemma I.4.23]{STV} for the proof of an analogous assertion for real valued
functions).
\end{remark}

\section{Variation on elementary sets}\label{Sec4}
First, motivated by \cite[Definition 6.1]{G}, we will introduce the definition
of a~variation over arbitrary intervals.

\smallskip

\begin{definition}\label{var-J}
Let $J$ be a bounded interval in $\R.$ We say that a finite set
\[
   D\,{=}\,\{\alpha_0,\alpha_1,\dots,\alpha_{\nu(D)}\}\subset J
\]
is a {\em generalized division of} $J$ if
\[
   \alpha_0<\alpha_1<\dots<\alpha_{\nu(D)}.
\]
The set of all generalized divisions of the interval $J$ is denoted by $\mathcal{D^*}(J).$

Let $f{:}\,\ab\,{\to}\,X$ and let $J$ be an arbitrary subinterval of $\ab.$
Then we define the variation of $f$ on $J$  by
\[
    \var_J f=\sup_{D\in\mathcal{D^*}(J)}
          \left\{\,\sum_{j=1}^{\nu(D)}\|f(\alpha_j)-f(\alpha_{j-1})\|_X\,\right\}.
\]
We say that $f$ is of bounded variation on $J$ if $\var_J f<\infty.$ In such a case,
we write $f\in BV(J,X).$ For convention we set also
\[
  \var_\emptyset f=0\mbox{ \ and \ }\var_{[c]} f=0\quad\mbox{for \ } c\in\ab.
\]
\end{definition}

\smallskip

\begin{remark}\label{P4.2}
It is easy to see that Definition~\ref{var-J} coincides with the definition of
the variation in the sense of Jordan if $J$ is a compact interval, that is, for
$f{:}\,\ab\,{\to}\,X$ and $J=[c,d]\subseteq\ab$ we have
\[
   \var_{[c,d]}f=\var_c^{\,d}f.
\]
For this reason, in the case of a compact interval, {\em we may always restrict
ourselves to the divisions containing their end points} (as defined in Section~2).

Moreover, it is easy to see that if $J$ is a bounded interval and $f\in BV(J,X),$ then
$f$ is bounded on $J.$
\end{remark}

\smallskip

\begin{proposition}
Let $f\,{:}\,\ab\to X$ and let $J_1$ and $J_2$ be subintervals of $\ab$ such that
$J_2\subseteq J_1$. Then
\[
   \var_{J_2}f\le\var_{J_1}f.
\]
In particular, if $J$ is a subinterval of $\ab$ and $f\in BV(J,X),$ then
$f\in BV(I,X)$ for every interval $I\subseteq J.$
\end{proposition}

\smallskip

The next theorem presents an equivalent formulation of the variation on arbitrary
intervals commonly found in literature.

\smallskip

\begin{theorem}\label{T4.4}
Let $f\,{:}\,\ab\,{\to}\,X$ and $c,d\in\ab,$ with $c<d.$
\begin{enumerate}[{\rm(i)}]
\item If $f\in BV([c,d),X),$ then
\quad $\var_{[c,d)}f=\lim_{\delta\to 0+}\var_c^{d{-}\delta}f=\sup_{t\in[c,d)}\var_c^tf.$

\item If $f\in BV((c,d],X),$ then
\quad $\var_{(c,d]}f=\lim_{\delta\to 0+}\var_{c{+}\delta}^{\,d}f=\sup_{t\in(c,d]}\var_t^{\,d}f.$

\item If $f\in BV((c,d),X),$ then
\quad $\var_{(c,d)}f=\lim_{\delta\to 0+}\var_{c{+}\delta}^{d-\delta}f.$
\end{enumerate}
\end{theorem}
\begin{proof} We will prove the assertion (i), the remaining ones follow in a similar way.

\smallskip

For a fixed $\delta>0,$ consider a division $D=\{\alpha_0,\alpha_1,\ldots,\alpha_{\nu(D)}\}$
of $[c,d{-}\delta].$ Of course, $D$ is also a generalized division of $[c,d)$ and hence
\[
   \sum_{j=1}^{\nu(D)}\|f(\alpha_j)-f(\alpha_{j-1})\|_X\le\var_{[c,d)}f\,.
\]
Thus, taking the supremum over all divisions of $[c,d-\delta],$ we get
$\var_c^{\,d{-}\delta}f\le\var_{[c,d)}f.$ Since this inequality holds for every
$\delta>0,$ it follows that
\[
    M:=\sup_{t\in[c,d)}\var_c^t f
      =\lim_{\delta\to 0+}\var_c^{d{-}\delta}f\le\var_{[c,d)}f\,.
\]
Now, assume that $\dis M<\var_{[c,d)}f.$ Then, for $\eps=\var_{[c,d)}f-M,$ there exists
a division $D\,{=}\,\{\alpha_0,\alpha_1,\ldots,\alpha_{\nu(D)}\}$ of $[c,d)$ such that
\[
   M=\var_{[c,d)}f-\varepsilon<\sum_{j=1}^{\nu(D)}\|f(\alpha_j)-f(\alpha_{j-1})\|_X
                            \le\var_c^{\alpha_{\nu(D)}}f\le M,
\]
a contradiction. This completes the proof of (i).
\end{proof}

\smallskip

Dealing with functions with values in a metric space, Chistyakov presents in \cite{Ch}
an extensive study of the properties of the variation over subsets of the real line.
Here, we call the reader's attention to a particular result (see \cite[Corollary 4.7]{Ch})
connecting the variation over arbitrary intervals and the Jordan variation over a compact
interval. This will be the content of Theorem~\ref{vars} whose proof is included for sake
of completeness. To this aim, the next lemma will be usefull.

\smallskip

\begin{lemma}\label{limits}
Let $f{:}\,\ab\,{\to}\,X,$  $a\le c<d\le b$ and $f\in BV((c,d),X).$
Then both limits $f(c+)$ and $f(d-)$ exist.
\end{lemma}
\begin{proof}
Let $\eps>0$ and an increasing sequence $\{t_n\}\subset (c,d)$ tending to $d$ be given.
By Theorem~\ref{T4.4}~(i) there is $\delta>0$ such that
\[
   0<\var_{[c,d)}f-\var_c^{d-\delta}f<\varepsilon.
\]
Choose $n_0\in\N$ in such a way that $t_n>d-\delta$ for every $n\ge n_0.$ Therefore,
for $n>m\ge n_0,$ we have
\begin{align*}
    \|f(t_n)-f(t_m)\|_{X}&\le\var_{t_m}^{t_n}f=\var_c^{t_n}f-\var_c^{t_m}f
  \\
    &\le\var_{[c,d)}f-\var_c^{d-\delta}f<\eps
\end{align*}
wherefrom the existence of the limit $f(d-)$ immediately follows.

The existence of the limit $f(c+)$ can be proved analogously.
\end{proof}

\smallskip

\begin{theorem}\label{vars}
Let $f{:}\,\ab\,{\to}\,X$ and $a\le c<d\le b.$
\begin{enumerate}[{\rm(i)}]
\item If $f\in BV([c,d),X)$, then $f(d-)$ exists and
\[
   \var_c^{\,d} f=\var_{[c,d)}f+\|\Delta^-f(d)\|_X.
\]
\item If $f\in BV((c,d],X)$, then $f(c+)$ exists and
\[
\var_c^{\,d} f=\var_{(c,d]}f+\|\Delta^+f(c)\|_X.
\]
\item If $f\in BV((c,d),X)$, then both limits $f(c+)$ and $f(d-)$ exist
and
\[
\var_c^{\,d} f=\var_{(c,d)}f+\|\Delta^+f(c)\|_X+\|\Delta^-f(d)\|_X.
\]
\end{enumerate}
\end{theorem}
\begin{proof}
The existence of all the needed limits follows by Lemma~\ref{limits}.

\smallskip

Assume that $f\in BV([c,d),X)$ and let $\eps>0$ and
$D=\{\alpha_0,\alpha_1,\dots,\alpha_{m+1}\}\in\mathcal D[c,d]$ be given.
We can choose $\xi\in[c,d]$ in such a way that
\[
   \alpha_m<\xi<d \quad\mbox{and}\quad \|f(d-)-f(\xi)\|_X<\eps.
\]
Consequently,
\begin{align*}
    &\sum_{j=1}^{m{+}1}\|f(\alpha_j)-f(\alpha_{j-1})\|_X
   \\
    &\qquad
     \le\sum_{j=1}^{m}\|f(\alpha_j)\,{-}\,f(\alpha_{j-1})\|_X+\|f(\xi)\,{-}\,f(\alpha_m)\|_X
     +\|f(d-)\,{-}\,f(\xi)\|_X+\|\Delta^-f(d)\|_X
   \\
    &\qquad<\var_c^{\,\xi}f+\eps+\|\Delta^-f(d)\|_X
     \le\var_{[c,d)}f+\eps+\|\Delta^-f(d)\|_X.
\end{align*}
As $D\in\mathcal D[c,d]$ and $\eps>0$ were arbitrarily chosen, we conclude that
\begin{equation}\label{r4.1}
   \var_c^{\,d} f\le\var_{[c,d)}f+\|\Delta^-f(d)\|_X.
\end{equation}

On the other hand, for any $\delta>0$ we have
\[
   \|f(d)-f(d-\delta)\|_X\le\var_{d-\delta}^{\,d} f=\var_c^{\,d}f-\var_c^{d-\delta}f.
\]
Hence, letting $\delta\to 0+$ we get
\[
   \|\Delta^-f(d)\|_X
   \le\var_c^{\,d}f-\lim_{\delta\to 0+}\var_c^{d-\delta}f=\var_c^{\,d}f-\var_{[c,d)}f
\]
wherefrom, due to \eqref{r4.1}, we conclude that
$\dis\var_c^{\,d}f=\var_{[c,d)}f+\|\Delta^-f(d)\|_X.$ This completes the proof of (i).

\smallskip

Similarly, we can prove the assertions (ii) and (iii).
\end{proof}

\smallskip

\begin{corollary}\label{C4.5}
Let $f{:}\,\ab\,{\to}\,X$ and $c,d\in\ab,$ with $c<d.$ Then the following assertions
are equivalent:
\begin{enumerate}[{\rm(i)}]
  \item  $f\in BV([c,d],X),$
  \item  $f\in BV((c,d],X),$
  \item  $f\in BV([c,d),X),$
  \item  $f\in BV((c,d),X).$
\end{enumerate}
\end{corollary}

\smallskip

\begin{remark}\label{P4.5}
In view Theorem~\ref{vars}, we can also observe that for $f{:}\,\ab\,{\to}\,X$ and $c\in\ab,$
we have
\[
   \lim_{\delta\to 0+}\var_{\,c{-}\delta}^{\,c{+}\delta}f=\|\Delta^-f(c)\|_X+\|\Delta^+f(c)\|_X
\]
provided the one-sided limits exist at the point $c$ (see \cite[Proposition I.2.8]{H1}).
Furthermore, by Theorem \ref{vars}, if $f\in BV(\ab,X)\cap C(\ab,X),$ then
\[
   \var_{[c,d)}f=\var_{(c,d)}f=\var_{(c,d]}f=\var_c^{\,d}f
   \quad\mbox{for \ } c,\,d\in\ab \quad\mbox{such that \ } c<d.
\]
\end{remark}

\smallskip

We need to extend the notion of a variation on intervals to elementary sets.

\smallskip

\begin{definition}
Let $E\subset\R$ be bounded. We say that $E$ is an {\em elementary set} if it is
a finite union of intervals.

\smallskip

A collection of intervals $\{J_k{:}\,k=1,\dots,m\}$ is called a {\em minimal
decomposition} of $E$ if $\dis E=\bigcup_{k=1}^m J_k$ and the union $J_k\cup J_\ell$
is not an interval whenever $k\ne\ell.$

\smallskip

If $S\subseteq\R,$ then $\mathcal E(S)$ stands for the set of all elementary subsets of $S.$
\end{definition}

\smallskip

Note that the minimal decomposition of an elementary set is uniquely determined. Moreover,
the intervals of such decomposition are pairwise disjoint.

\smallskip

Having this in mind, we extend the notion of a variation over elementary sets as follows.

\smallskip

\begin{definition}\label{D4.8}
Given a function $f{:}\,\ab\,{\to}\,X$ and an elementary subset $E$ of $\ab,$ the variation
of $f$ over $E$ is
\[
   \var(f,E)=\sum_{k=1}^m\var_{J_k}f,
\]
where $\{J_k{:}\,k=1,\dots,m\}$ is the minimal decomposition of $E.$
\end{definition}

\smallskip

It is worth highlighting that if $f\in BV(\ab,X)\cap C(\ab,X),$ then
$\var(f,\,\cdot\,)$ defines a finitely additive measure on $\mathcal E(\ab).$
More precisely, we have
\[
    \var(f,E)\le\var_a^bf  \quad\mbox{for any \ } E\in\mathcal{E}(\ab)
\]
and
\[
   \var(f,E_1\cup E_2)=\var(f,E_1)+\var(f,E_2)
\]
whenever $E_1,\,E_2\in\mathcal E(\ab)$ and $E_1\cap E_2=\emptyset.$

\smallskip

\begin{remark}\label{P4.9}
Let us note that in \cite{G} Definition~\ref{var-J} is applied also to arbitrary
subsets $E$ of $\ab.$ Unfortunately, such a definition is not convenient
for our purposes, as the variation would lose the additivity property even
for continuous functions. Indeed, let $a<c<d<b$ and $E=[a,c]\cup[d,b].$  Then,
according to such a definition we would have
\[
    \var_Ef\ge\var_a^cf+\var_d^bf+\|f(d)-f(c)\|>\var_{[a,c]}f+\var_{[d,b]}f
\]
whenever $f(d)\ne f(c).$ This is why for elementary subsets of $\ab$ we will define
the variation in a different way than that used by Gordon in \cite{G}.
\end{remark}

\smallskip

The following lemma will give us the crucial tool for proving our main result.
It is analogous to Lemma from \cite{Lewin}. However, instead of the Lebesgue measure
it works with the variation of the given function over elementary sets. We say that
a sequence $\{A_n\}$ of subsets of $\ab$ is {\em contracting} if $A_{n+1}\subseteq A_n$
holds for $n\in\N.$

\smallskip

\begin{lemma}\label{measure}
Let $f\in BV(\ab,X)\cap C(\ab,X)$ and let $\{A_n\}$ be a contracting sequence
of subsets of $\ab$ such that $\dis\bigcap_n A_n=\emptyset.$ For $n\in\N$
put
\begin{equation}\label{alpha}
   v_n=\sup\{\,\var(f,E){:}\,E\in\mathcal E(A_n)\,\}.
\end{equation}
Then \quad $\lim_{n\to\infty}v_n=0.$
\end{lemma}
\begin{proof} First, notice that, by Definition \ref{var-J}, $v_n=0$ whenever
$A_n=\emptyset.$

\smallskip

Let us assume that $v_n$ does not tend to zero. Since $\{v_n\}$ is
decreasing, this means that there is $\eps>0$ such that
\begin{equation}\label{r4.2}
   v_n>\eps\quad\mbox{for all \ } n\in\N.
\end{equation}
Consequently, for each $n\in\N,$ there is an $E_n\in\mathcal E(A_n)$ such that
\begin{equation}\label{En}
    \var(f,E_n)>v_n-\frac{\eps}{2^n}\,.
\end{equation}
Since $f$ is continuous, we can assume that $E_n$ is closed (cf. Remark~\ref{P4.5}).
Now, let $H_n=\dis\bigcap_{j=1}^nE_j$ for $n\in\N.$ Clearly, $H_n$ is closed and
$H_n\subseteq A_n$ holds for any $n\inn\N.$ We shall show that
\begin{equation*}\label{r4.3}
    H_n\ne\emptyset\quad\mbox{for any \ } n\in\N.
\end{equation*}

\smallskip

Indeed, let $n\inn\N$ and $M\inn\mathcal E(A_n{\setminus}E_n)$ be given. Then
$M\cup E_n\in\mathcal{E}(A_n)$ and, due to the additivity of the variation, we can
see that
\[
   \var(f,M)+\var(f,E_n)=\var(f,M\cup E_n)\le v_n
   \quad\mbox{for all \ } n\in\N.
\]
Therefore, due to \eqref{En}, we have $\var(f,M)<\frac{\eps}{2^n}.$  Obviously,
\[
   E=\bigcup_{j{=}1}^n (E{\setminus}\,E_j)
   \quad\mbox{for any \ } E\in\mathcal{E}(A_n{\setminus}\,H_n),
\]
where $(E\setminus E_j)\in\mathcal{E}(A_n{\setminus}E_j)
\subseteq\mathcal{E}(A_j{\setminus}E_j)$ for  $j=1,\dots,n.$ Hence
\[
   \var(f,E)\le\sum_{j{=}1}^n\var(f,E{\setminus}\,E_j)
            <\sum_{j{=}1}^n\dfrac{\eps}{2^j}<\eps
   \mbox{ \ for every $E\in\mathcal E(A_n{\setminus}\,H_n)$}.
\]
This, together with \eqref{r4.2}, implies that there must exist $E\in\mathcal E(H_n)$
with $\var(f,E)>\eps,$ wherefrom we conclude that $H_n\ne\emptyset.$

Since $\{H_n\}$ is a contracting sequence of non-empty, closed and bounded sets,
using Cantor's intersection theorem we get
\[
   \bigcap_n A_n\supset\bigcap_n H_n\ne\emptyset\,.
\]
Of course, this contradicts our hypothesis $\dis\bigcap_n A_n=\emptyset.$ As a result,
$\dis\lim_{n\to\infty}v_n=0$ and this completes the proof.
\end{proof}

\section{Integration over elementary sets}
To our knowledge, unlike the Lebesgue-Stieltjes integral, up to now the KS-integral
over sets that need not be just the compact intervals has not been discussed
in literature. In the case of elementary subsets of $\ab$ the following definition
turned out to be useful for our purposes.

\skipaline

\begin{definition}\label{D5.1}
Let $F{:}\,\ab\,{\to}\,L(X),$ \ $g{:}\,\ab\,{\to}\,X$ and an elementary subset $E$
of $\ab$ be given. The Kurzweil-Stieltjes integral (or shortly KS-integral) of $g$
with respect to $F$ over $E$ is given by
\[
\int_E\dd[F]g=\int_a^b\dd[F]\,(g\,\chi_E)
\]
provided the integral on the right-hand side exists.

\smallskip

Symmetrically, we define the KS-integral of $F$ with respect to $g$ over $E$ by
\[
\int_EF\dd[g]=\int_a^b(F\,\chi_E)\,\dd[g]
\]
provided the integral on the right-hand side exists.
\end{definition}

\smallskip

According to Definition~\ref{D5.1} the existence of the integral $\dis\int_E\dd[F]\,g$
means (see Section~1) that there is an $I\in X$ such that for every $\eps\,{>}\,0$
we can find a gauge $\delta$ on $\ab$ such that
\[
   \Big\|S(\dd F,g\,\chi_E,P)-I\Big\|_X<\eps
   \quad\mbox{for all \ $\delta$-fine tagged divisions $P$ of $\ab$}.
\]

\smallskip

Due to Definition~\ref{D5.1} the basic properties of the KS-integral mentioned below
are immediate consequences of what is known for the abstract Kurzweil-Stieltjes integral.

\smallskip

\begin{proposition}\label{linear}
Let $E$ be an elementary subset of $\,\ab.$ Assume that $F{:}\ab{\to}L(X)$
and \ $g_{j}\,{:}\,\ab\,{\to}X,$ $j=1,2,$ are such that the integrals
\[
    \int_E\dd[F]\,g_1 \quad\mbox{and}\quad \int_E\dd[F]\,g_2
\]
exist. Then the integral \ $\dis\int_E\dd[F]\,(c_1\,g_1+c_2\,g_2)$ \ exists and
\[
  \int_E\dd[F]\,(c_1\,g_1+c_2\,g_2)=c_1\int_E\dd[F]\,g_1+c_2\int_E\dd[F]\,g_2
  \quad\mbox{for every \ } c_1,c_2\in\R.
\]

Symmetrically, if $F_j\,{:}\,\ab\,{\to}\,L(X),$ $j=1,2,$ and $g\,{:}\,\ab\,{\to}\,X$
are such that the integrals
\[
    \int_E\dd[F_1]\,g \quad\mbox{and}\quad \int_E\dd[F_2]\,g
\]
exist, then the integral \ $\dis\int_E\dd[c_1\,F_1+c_2\,F_2]\,g$ \ exists and
\[
  \int_E\dd[c_1\,F_1+c_2\,F_2]\,g=c_1\int_E\dd[F_1]\,g+c_2\int_E\dd[F_2]\,g
  \quad\mbox{for all \ } c_1,c_2\in\R.
\]
\end{proposition}

\smallskip

\begin{remark}
Note that, if $E$ is an elementary subset of $\ab$ and $g\,{:}\,\ab\,{\to}\,X$
is a function such that $g=0$ on $E,$ then
\[
   \int_{E}\dd[F]\,g=0
\]
for every $F\,{:}\,\ab\,{\to}\,L(X).$

Consequently, due to Proposition \ref{linear}, if $F\,{:}\,\ab\,{\to}\,L(X)$ and
$g\,{:}\,\ab\,{\to}\,X$ are such that $\dis\int_{E}\dd[F]\,g$ exists, then for any
function $h\,{:}\,\ab\,{\to}\,X$ which coincides with $g$ on $E,$ we have
\[
   \int_{E}\dd[F]\,h=\int_{E}\dd[F]\,g.
\]
\end{remark}

\smallskip

\begin{theorem}\label{add}
Let $E_{1}$ and $E_{2}$ be elementary subsets of \ $\ab$ such that
$E_{1}\cap E_{2}=\emptyset.$ Assume that $F\,{:}\,\ab\,{\to}\,L(X)$ and
$g\,{:}\,\ab\,{\to}\,X$ are such that both the integrals
$\dis\int_{E_{j}}\dd[F]\,g,$ $j=1,2,$ exist. Then the integral
$\dis\int_{E_1\cup E_2}\dd[F]\,g$ exists and
\[
   \int_{E_1\cup E_2}\dd[F]\,g=\int_{E_{1}}\dd[F]\,g+\int_{E_{2}}\dd[F]\,g.
\]
\end{theorem}
\begin{proof}
Since $g\,(\chi_{E_{1}\cup E_{2}})=g\,\chi_{E_1}\,{+}\,g\,\chi_{E_{2}},$
by Definition \ref{D5.1} and \cite[Proposition 6]{Sch1} we have
\[
  \int_{E_{1}}\dd[F]\,g+\int_{E_{2}}\dd[F]\,g
  =\int_a^b\dd[F]\,(g\,\chi_{E_1})
   +\int_a^b\dd[F]\,(g\,\chi_{E_{2}})=\int_a^b\dd[F]\,(g\,\chi_{E_1\cup E_2}),
\]
which proves the result.
\end{proof}

\smallskip

The following existence result is a consequence of \cite[Proposition 15]{Sch1}.

\begin{proposition}\label{existence}
Let $F\in BV(\ab,L(X))$ and $g\in G(\ab,X).$ If $E$ is an elementary subset
of $\ab,$ then the integral $\int_E\dd[F]\,g$ exists.
\end{proposition}
\begin{proof} It is enough to observe that $g\,\chi_E\,{:}\,\ab\,{\to}\,X$ is
a regulated function whenever $g\in G(\ab,X)$ and $E$ is an elementary subset
of $\ab.$
\end{proof}

\smallskip

The estimates presented in the sequel are, in a sense, the analogues of \cite[Proposition 10]{Sch1}.

\begin{theorem}\label{int-inequality}
Let $J$ be a subinterval of $\ab$ and let $c=\inf J$ and $d=\sup J$, with $c<d$.
Assume that $F\,{:}\,\ab\,{\to}\,L(X),$ with $F\in BV(J,L(X)),$ and
$g\,{:}\,\ab\,{\to}\,X$ are such that the integral $\dis\int_{J}\dd[F]\,g$
exists.
\begin{enumerate}[{\rm (i)}]
\item If $J=(c,d),$ then
$\dis\Big\|\int_{(c,d)}\dd[F]\,g\Big\|_X\le\var_{(c,d)}F\,\Big(\sup_{t\in(c,d)}\|g(t)\|_X\Big).$

\item If $J=[c,d)$ and $F(c-)$ exists, then
\begin{equation*}\label{int[c,d)}
  \Big\|\int_{[c,d)}\dd[F]\,g\Big\|_X
  \le\var_{[c,d)}F\,\Big(\sup_{t\in[c,d)}\|g(t)\|_X\Big)
                                  +\|\Delta^-F(c)\|_{L(X)}\,\|g(c)\|_X.
\end{equation*}

\item If $J=(c,d]$ and $F(d+)$ exists, then
\begin{equation*}\label{int(c,d]}
  \Big\|\int_{(c,d]}\dd[F]\,g\Big\|_X
  \le\var_{(c,d]}F\,\Big(\sup_{t\in(c,d]}\|g(t)\|_X\Big)
                                      +\|\Delta^+F(d)\|_{L(X)}\,\|g(d)\|_X.
\end{equation*}

\item If $J=[c,d],$ and both $F(c-),\,F(d+)$ exist, then
\begin{align*}\label{int[c,d]}
  &\hskip-6mm\Big\|\int_{[c,d]}\hskip-1mm\dd[F]\,g\Big\|_X
\\[1mm]
  &\le\var_{[c,d]}F\left(\sup_{t\in[c,d]}\|g(t)\|_X\right)
	{+}\,\|\Delta^-F(c)\|_{L(X)}\|g(c)\|_X\,{+}\,\|\Delta^+F(d)\|_{L(X)}\|g(d)\|_X.
\end{align*}
\end{enumerate}
\end{theorem}
\begin{proof} Since the integral $\dis\int_J\dd[F]\,g$ exists, for any
$\eps>0$ there is a gauge $\delta$ on $\ab$ such that
\begin{equation}\label{intJ}
   \left\|S(\dd F,g\,\chi_J,P)-\int_J\dd[F]\,g\right\|_X<\eps
\end{equation}
for every $\delta$-fine tagged division $P$ of $\ab.$

\smallskip

Let $\delta^*$ be a gauge on $\ab$ such that $\delta^*\le\delta$ on $\ab$ and
\begin{equation}\label{c,d-special}
  \delta^*(t)<\min\{|t-c|,|t-d|\}
  \quad\mbox{for \ }t\in\ab{\setminus}\,\{c,d\}.
\end{equation}
Considering a $\delta^*$-fine tagged division $P$ of $\ab,$ put
$P_J\,{=}\,\{(\tau,I)\in P{:}\,I\,{\cap}\,[c,d]\,{\ne}\,\emptyset\}.$
Let $a\,{<}\,c\,{<}\,d\,{<}\,b.$ Denoting
$P_J\,{=}\,\{(\xi_j,[\beta_{j{-}1},\beta_j]){:}\,j\,{=}\,1,\dots,\nu(P_J)\},$
without loss of generality we may assume
\begin{equation}\label{P}
   \beta_0<c=\beta_1<\beta_2<\dots<\beta_{\nu(P_J)-1}=d<\beta_{\nu(P_J)}.
\end{equation}
It is easy to see that $\xi_1=\xi_2=c$ and $\xi_{\nu(P_J)}=\xi_{\nu(P_J){-}1}=d$
must hold.

\skipaline

\noindent (i)\quad For $J=(c,d),$ we have
\[
  \begin{split}
    S(\dd F,g\,\chi_{(c,d)},P)&=S(\dd F,g\,\chi_{(c,d)},P_J)
   \\
	&=\sum_{j=1}^{\nu(P_J)}[F(\beta_j)\,{-}\,F(\beta_{j-1})](g\,\chi_{(c,d)})(\xi_j)
     =\hskip-1mm\sum_{j=3}^{\nu(P_J)-2}[F(\beta_j)\,{-}\,F(\beta_{j-1})]g(\xi_j).
  \end{split}
\]
Having in mind that $\{\beta_2,\dots,\beta_{\nu(P_J)-2}\}$ is a division of $(c,d),$
we conclude that
\[
  \|S(\dd F,g\,\chi_{(c,d)},P)\|_X\le\var_{(c,d)}F\,\Big(\sup_{t\in(c,d)}\|g(t)\|_X\Big)
\]
and consequently, due to \eqref{intJ},
\[
 \Big\|\int_{(c,d)}\dd[F]\,g\Big\|_X
 \le\var_{(c,d)}F\,\Big(\sup_{t\in (c,d)}\|g(t)\|_X\Big)+\eps.
\]
Since $\eps>0$ can be arbitrary, the desired inequality is true.

\skipaline

\noindent (ii)\quad Let $J=[c,d)$ and let $g(c)\ne 0.$ Since $F(c-)$ exists, we
can choose $\eta>0$ such that
\[
   \|F(c-)-F(s)\|_{L(X)}<\frac{\eps}{\|g(c)\|_X} \quad\mbox{for \ }c-\eta<s<c.
\]
Moreover, we can assume that the gauge $\delta^*$ as in \eqref{c,d-special} is such
that $\delta^*(c)<\eta.$

Given a $\delta^*$-fine tagged division $P$ of $\ab$ and considering $P_J$ as in \eqref{P},
we have
\[
   S(\dd F,g\,\chi_{[c,d)},P)=\sum_{j=2}^{\nu(P_J)-2}
                       [F(\beta_j)-F(\beta_{j-1})]\,g(\xi_j)+[F(c)-F(\beta_0)]\,g(c).
\]
Hence
\begin{align*}
    \|S(\dd F,g\,\chi_{[c,d)},P)\|_X
    &\le\var_{[c,d)}F\,\Big(\sup_{t\in[c,d)}\|g(t)\|_X\Big)
                                                +\|F(c)-F(\beta_0)\|_{L(X)}\,\|g(c)\|_X
  \\[2mm]
	&\le\var_{[c,d)}F\,\Big(\sup_{t\in[c,d)}\|g(t)\|_X\Big)
                                                +\|\Delta^-F(c)\|_{L(X)}\,\|g(c)\|_X
  \\
	&\qquad\quad+\|F(c-)-F(\beta_0)\|_{L(X)}\,\|g(c)\|_X
  \\[2mm]
	&\le\var_{[c,d)}F\,\Big(\sup_{t\in [c,d)}\|g(t)\|_X\Big)
                                            +\|\Delta^-F(c)\|_{L(X)}\,\|g(c)\|_X+\eps.
\end{align*}
In view of this and applying \eqref{intJ}, for $J=[c,d)$ we obtain
\[
  \Big\|\int_{[c,d)}F\,\dd[g]\Big\|_X
  \le\var_{[c,d)}F\Big(\sup_{t\in[c,d)}\|g(t)\|_X\Big)
                                      +\|\Delta^-F(c)\|_{L(X)}\,\|g(c)\|_X+2\,\eps,
\]
which proves (ii).

\skipaline

\noindent
(iii)\quad In order to prove the desired inequality for $J=(c,d],$ assume $g(d)\ne 0$
and choose $\gamma>0$ such that
\[
   \|F(s)-F(d+)\|_{L(X)}<\frac{\eps}{\|g(d)\|_X}\quad\mbox{for \ }d<s<d+\gamma.
\]
Without loss of generality we can assume that the gauge $\delta^*$ in \eqref{c,d-special}
is such that $\delta^*(d)<\gamma.$  For a $\delta^*$-fine tagged division $P$ of $\ab,$
considering $P_J$ as in \eqref{P}, we have
\[
  S(\dd F,g\,\chi_{(c,d]},P)
  =\sum_{j=3}^{\nu(P_J)-1}
             [F(\beta_j)-F(\beta_{j-1})]\,g(\xi_j)+[F(\beta_{\nu(P_J)})-F(d)]\,g(d),
\]
wherefrom
\begin{align*}
   \|S(\dd F,g\,\chi_{(c,d]},P)\|_X
   &\le\var_{(c,d]}F\Big(\sup_{t\in(c,d]}\,\|g(t)\|_X\Big)+\|\Delta^
                                                          +F(d)\|_{L(X)}\,\|g(d)\|_X
 \\
   &\qquad\quad+\|F(\beta_{\nu(P_J)}-F(d+))\|_{L(X)}\|g(d)\|_X
 \\
   &\le\var_{(c,d]}F\Big(\sup_{t\in(c,d]}\,\|g(t)\|_X\Big)
                                           +\|\Delta^+F(d)\|_{L(X)}\,\|g(d)\|_X+\eps.
\end{align*}
Therefore, by \eqref{intJ} we obtain
\[
  \Big\|\int_{(c,d]}\dd[F]\,g\Big\|_X\le
  \var_{(c,d]}F\Big(\sup_{t\in(c,d]}\,\|g(t)\|_X\Big)+\|\Delta^+F(d)\|_{L(X)}\,\|g(d)\|_X+2\,\eps
\]
and (iii) follows.

\skipaline

\noindent(iv)\quad Similarly to the above, let us consider a $\delta^*$-fine tagged division $P$
of $\ab,$ where the gauge $\delta^*$ satisfies \eqref{c,d-special}, $\delta^*(c)<\eta$
and $\delta^*(d)<\gamma.$ We can see that
\[
\begin{split}
   S(\dd F,g\,\chi_{[c,d]},P)
   &=\sum_{j=2}^{\nu(P_J)-1}[F(\beta_j)-F(\beta_{j-1})]\,g(\xi_j)
  \\
   &\qquad +[F(c)-F(\beta_0)]\,g(c)+[F(\beta_{\nu(P_J)})-F(d)]\,g(d),
\end{split}
\]
where $P_J$ is defined as in \eqref{P}. Thus
\[
\begin{split}
    &\|S(\dd F,g\,\chi_{[c,d]},P)\|_X
  \\&\quad
     {\le}\var_{[c,d]}F\Big(\sup_{t\in [c,d]}\,\|g(t)\|_X\Big)
       \,{+}\,\|\Delta^-F(c)\|_{L(X)}\,\|g(c)\|_X
       \,{+}\,\|\Delta^+F(d)\|_{L(X)}\,\|g(d)\|_X
       \,{+}\,2\,\eps,
\end{split}
\]
which, together with \eqref{intJ}, leads to the desired inequality.

\smallskip

The case when $c=a$ and/or $d=b$ can be proved in a similar way, recalling that,
by convention, $\Delta^-F(a)=\Delta^+F(b)=0.$
\end{proof}

\smallskip

In order to better characterize the integral over elementary sets, first we investigate
the integral over arbitrary intervals. The case when $E$ is just one point corresponds
to the following result borrowed from \cite[Lemma 12]{Sch1}.

\smallskip

\begin{proposition}\label{prop-point}
Let $F\in(\mathcal{B})G([a,b],L(X))$ and $g\,{:}\,\ab\,{\to}\,X$ be given. If
\ $\tau\in\ab,$ then the integral $\dis\int_{[\tau]}\dd[F]\,g$ exists and one
of the following equalities is true:
\begin{equation}\label{point}\hskip-1mm
  \left.\begin{array}{l}\dis
    \int_{[a]}\dd[F]\,g=\int_a^b\dd[F]\,(g\,\chi_{[a]})
    =\lim_{t\to a+}F(t)\,g(a)\,{-}\,F(a)\,g(a),
   \\[5mm]\dis
    \int_{[\tau]}\dd[F]\,g=\int_a^b\dd[F]\,(g\,\chi_{[\tau]})
    =\lim_{t\to\tau+}F(t)\,g(\tau)\,{-}\lim_{t\to\tau-}F(t)\,g(\tau)
    \mbox{\ if \ } \tau\in(a,b),
   \\[5mm]\dis
    \int_{[b]}\dd[F]\,g=\int_a^b\dd[F]\,(g\,\chi_{[b]})
    =F(b)\,g(b)\,{-}\lim_{t\to b-}F(t)\,g(b).
\end{array}\hskip-1mm\right\}
\end{equation}
\end{proposition}

\smallskip

An important question which arises regarding Definition \ref{D5.1} is whether the integral
over compact intervals coincides with the abstract Kurzweil-Stieltjes integral introduced
in Section~1. As we will see, the equality need not be true in general. However, given
functions $F\,{:}\,\ab\,{\to}\,L(X)$ and $g{:}\,\ab\,{\to}\,X,$ the equality
\[
  \int_{[a,b]}\dd[F]\,g=\int_{a}^{b}\dd[F]\,g
\]
holds provided one of the integrals exists.

\smallskip

\begin{theorem}\label{compact}
Let $F\in(\mathcal{B})G([a,b],L(X)),$ \ $g{:}\,\ab\,{\to}\,X$ and $a<c<d<b.$
Then the integral $\dis\int_{[c,d]}\dd[F]\,g$ exists if and only if
$\dis\int_c^{\,d}\dd[F]\,g$ exists. Moreover,
\begin{align}\label{[c,d]}
     \int_{[c,d]}\dd[F]\,g
   &=\int_{c}^{d}\dd[F]\,g+F(c)\,g(c)-\lim_{t\to c-}F(t)\,g(c)
  \\\nonumber
   &\quad+\lim_{t\to d+}F(t)\,g(d)-F(d)\,g(d).
\end{align}
\end{theorem}
\begin{proof}
Obviously,
\begin{align*}
     \int_a^b\dd[F]\,(g\,\chi_{[c,d]})
   &=\int_a^c\dd[F]\,(g\,\chi_{[c,d]})+\int_c^{\,d}\dd[F]\,(g\,\chi_{[c,d]})
                                                    +\int_d^b\dd[F]\,(g\,\chi_{[c,d]})
  \\
   &=\int_a^c\dd[F]\,(g\,\chi_{[c]})+\int_c^{\,d}\dd[F]\,g+\int_d^b\dd[F]\,(g\,\chi_{[d]}).
\end{align*}
Hence, the proof follows from \eqref{point}.
\end{proof}

\smallskip

\begin{remark}
Recalling that $BV([a,b],X)\subset G([a,b],X),$ formula \eqref{[c,d]} is valid when
$F\in BV([a,b],L(X)).$ In the particular case $X=\R$ this fact corresponds to the result
given by Saks in \cite[Theorem VI.(8.1) (p.~208)]{S} which states the following:
\begin{quote}
  \textit{A finite function $f$ integrable in the Lebesgue-Stieltjes $(LS)$ sense on
  an interval $[c,d]$ with respect to a function of bounded variation $\varphi,$ is also
  integrable in the Perron-Stieltjes $(PS)$ sense and we have}
\begin{align*}
     &(PS)\int_{c}^{d}\dd\left[\varphi(s)\right]\,f(s)
   \\
     &\quad=(LS)\int_{[c,d]}\dd\left[\varphi(s)\right]\,f(s)
            -\Delta^{-}\varphi(c)\,f(c)
			-\Delta^{+}\varphi(d)\,f(d).
  \end{align*}
\end{quote}
\end{remark}

\smallskip

\begin{theorem}\label{open}
Let $F\in(\mathcal{B})G([a,b],L(X)),$
\ $g{:}\,[a,b]\to X,$ $c,\,d\in\ab$ and $c<d.$
Then the integral \ $\dis\int_{(c,d)}\dd[F]\,g$ exists if and only if the integral
$\dis\int_{c}^{d}\dd[F]\,g$ exists. Moreover,
\begin{equation}\label{(c,d)}
\left.\begin{array}{l}\dis\hskip-8mm
  \int_{(c,d)}\dd[F]\,g
 \\[5mm]\dis
  =F(c)\,g(c)-\lim_{t\to c+}F(t)\,g(c)+\int_{c}^{d}\dd[F]\,g
   +\lim_{t\to d-}F(t)g(d)-F(d)g(d).
\end{array}\quad\right\}
\end{equation}
\end{theorem}
\begin{proof}
By Proposition \ref{prop-point}, the integrals $\dis\int_{[c]}\dd[F]\,g$ and
$\dis\int_{[d]}\dd[F]\,g$ exist. Since
\[
   g\,\chi_{[c,d]}=g\,\chi_{[c]}+g\,\chi_{(c,d)}+g\,\chi_{[d]},
\]
it follows that the integral over $(c,d)$ exists if and only if the integral over
$[c,d]$ exists. In addition, by Theorem \ref{add} we obtain
\[
    \int_{[c,d]}\dd[F]\,g
   =\int_{[c]}\dd[F]\,g+\int_{(c,d)}\dd[F]\,g+\int_{[d]}\dd[F]\,g,
\]
which, together with \eqref{point} and Theorem~\ref{compact}, completes the proof.
\end{proof}

\smallskip

In a similar way, applying Proposition \ref{prop-point} and dealing with
characteristic functions, we can derive the following expressions for the integral
over half-open intervals.

\smallskip

\begin{theorem}\label{half}
Let $F\in(\mathcal{B})G([a,b],L(X)),$ \ $g{:}\,[a,b]\to X,$ \ $c,\,d\in\ab$
and $c<d.$ Then the following assertions are true:
\begin{enumerate}[$(i)$]
\item The integral $\dis\int_{[c,d)}\dd[F]\,g$ exists if and only if the integral
$\dis\int_{c}^{d} \dd[F]\,g$ exists.
Moreover,
\begin{equation}\label{[c,d)}
\left.\begin{array}{l}\dis\hskip-8mm
     \int_{[c,d)} \dd[F]\,g
    \\\dis\hskip-4mm
     =F(c)\,g(c)\,{-}\lim_{t\to c-}F(t)\,g(c)\,{+}\int_{c}^{d}\dd[F]\,g
                                \,{+}\lim_{t\to d-}F(t)\,g(d)-F(d)\,g(d).
\end{array}\right\}
\end{equation}

\item The integral $\dis\int_{(c,d]}\dd[F]\,g$ exists if and only if the integral
$\dis\int_{c}^{d}\dd[F]\,g$ exists. Moreover,
\begin{equation}\label{(c,d]}
\left.\begin{array}{l}\dis\hskip-8mm
    \int_{(c,d]}\dd[F]\,g
   \\\dis\hskip-4mm
    =F(c)\,g(c)\,{-}\lim_{t\to c+}F(t)\,g(c)\,{+}\int_{c}^{d}\dd[F]\,g
                                 \,{+}\lim_{t\to d+}F(t)\,g(d)-F(d)\,g(d).
\end{array}\right\}
\end{equation}
\end{enumerate}
\end{theorem}

\smallskip

\begin{remark}
According to Theorems \ref{compact}, \ref{open} and \ref{half},
for $F\in(\mathcal{B})G([a,b],L(X)),$ $g{:}\,\ab\,{\to}\,X$ and
$a\le c<d\le b,$
if one of the integrals
\begin{align*}
   \int_{[c,d]}\dd[F]\,g,\,\,\int_{[c,d)}\dd[F]\,g,\,\,\int_{(c,d]}\dd[F]\,g,
   \,\,\int_{(c,d)}\dd[F]\,g \mbox{ \ and \ } \int_c^{\,d}\dd[F]\,g
\end{align*}
exists, then all the others exist as well.

If we assume in addition $F\in C([a,b],L(X)),$ then
\[
   \int_{[c,d)}\dd[F]\,g=\int_{(c,d]}\dd[F]\,g=\int_{(c,d)}\dd[F]\,g=\int_c^{\,d}\dd[F]\,g
   =\int_{[c,d]}\dd[F]\,g.
\]
\end{remark}

\smallskip

We are now ready to deduce an expression for the integral over elementary sets.

\begin{theorem}\label{elementary}
Let $E$ be an elementary subset of $\ab.$ If $F\in(\mathcal{B})G([a,b],L(X))$ and
$g{:}\,\ab\,{\to}\,X$ are such that the integral $\dis\int_E\dd[F]\,g$ exists, then
\begin{equation}\label{int-elementary}
   \int_{E}\dd[F]\,g=\sum_{k=1}^{m}\int_{J_k}\dd[F]\,g,
\end{equation}
where $\{J_k{:}\,k=1\dots,m\}$ is the minimal decomposition of $E.$
\end{theorem}
\begin{proof}
For $k=1,\dots,m,$ let $c_k$ and $d_k$ be such that $\left[c_k,d_k\right]$ is
the closure of $J_k.$  By the hypothesis the integral
\[
    \int_{E}\dd[F]\,g:=\int_a^b\dd[F]\,(g\,\chi_{E})
\]
exists, and hence, by \cite[Proposition 8]{Sch1}, so does the integral
$\int_{c_k}^{d_k}\dd[F]\,(g\,\chi_{E})$. Moreover,
\[
  \int_{c_k}^{d_k}\dd[F]\,(g\,\chi_{E})=\int_{c_k}^{d_k}\dd[F]\,(g\,\chi_{J_k})
  \quad\mbox{for \ }k=1,\dots,m.
\]

\smallskip

Notice that, for each $k=1,\dots,m,$ the integrals
\[
   \int_{a}^{c_k}\dd[F]\,(g\,\chi_{J_k})
   \quad\mbox{and}\quad
   \int_{d_k}^{b}\dd[F]\,(g\,\chi_{J_k})
\]
are either zero or integrals over one point sets (which exist by
Proposition~\ref{prop-point}). Therefore, the integral $\int_{J_k}\dd[F]\,g$ exists.

\smallskip

Having in mind that the intervals of minimal decomposition are pairwise disjoint, \eqref{int-elementary}
follows from Theorem \ref{add}.
\end{proof}

\smallskip

As a consequence of Definition \ref{D4.8}, Theorem \ref{int-inequality} and Theorem \ref{elementary}
we have the following.

\smallskip

\begin{corollary}\label{estimate-elementary}
Let $E$ be an elementary subset of $\ab.$ If $F\in BV(\ab,L(X))\cap C(\ab,L(X))$ and $g{:}\,\ab\,{\to}\,X$ are
such that the integral $\int_{E}\dd[F]\,g$ exists, then
\[
   \left\|\int_{E}\dd[F]\,g\right\|_X\leq \var(F,E)\,\Big(\sup_{t\in E}\|g(t)\|_X\Big).
\]
\end{corollary}

\smallskip

Theorem \ref{elementary} also means that once the integral over an elementary set exists
so does the integral over subintervals of the minimal decomposition.
This implies immediately that the following assertions hold.

\smallskip

\begin{corollary}
Let $E$ be an elementary subset of $\ab.$ If $F\in(\mathcal{B})G([a,b],L(X))$ and $g{:}\,\ab\,{\to}\,X$
are such that the integral $\dis\int_E\dd[F]\,g$ exists, then the integral $\dis\int_T\dd[F]\,g$ exists for every
elementary subset $T$ of $\ab,$ with $T\subseteq E.$
\end{corollary}

\smallskip

\begin{corollary}
Let $E_{1}$ and $E_{2}$ be elementary subsets of $[a,b],$
$F\in(\mathcal{B})G([a,b],L(X))$ and $g{:}\,[a,b]\,{\to}\,X.$ If both the integrals
\[
   \int_{E_1}\dd[F]\,g \quad\mbox{and} \int_{E_2}\dd[F]\,g
\]
exist, then the integral $\dis\int_{E_1\cup E_2}\dd[F]\,g$ exists and
\begin{equation}\label{additivity}
    \int_{E_1\cup E_2}\dd[F]\,g
   =\int_{E_1}\dd[F]\,g+\int_{E_2}\dd[F]\,g-\int_{E_1\cap E_2}\dd[F]\,g.
\end{equation}
Similarly, if the integral $\dis\int_{E_1\cup E_2}\hskip-3mm\dd[F]\,g$ exists, then both
the integrals $\dis\int_{E_1}\hskip-3mm\dd[F]\,g$ and $\dis\int_{E_2}\hskip-3mm\dd[F]\,g$
exist and the equality \eqref{additivity} holds.
\end{corollary}

\smallskip

\begin{remark}
Of course, results analogous to those given in this section for integrals of the form
\quad $\dis\int_{E}\dd[F]\,g$ \quad can be obtained in a similar way also for the symmetrical
case \quad $\dis\int_{E}F\,\dd[g].$
\end{remark}

\section{Bounded convergence theorem}

Now we can prove our main result, which is Theorem \ref{Osgood}. To this aim, we will
first treat the case when the integrator is a function of bounded variation which
is also continuous. This is the content of the following theorem whose proof uses
the integration over elementary sets discussed in the previous section.

\smallskip

\begin{theorem}\label{Osgood-cont}
Let $F\in BV(\ab,L(X))\cap C(\ab,L(X))$ and assume that the sequence $\{g_n\}\subset G(\ab,X)$
is such that
\[
   \lim_{n\to\infty}g_n(t)=0\quad \mbox{for \ } t\in\ab
\]
and
\[
   \|g_n\|_\infty\le K<\infty \quad\mbox{for \ } n\in\N.
\]
Then the integral $\dis\int_a^b\dd[F]\,g_n$ exists for each $n\in\N$ and
\begin{equation}\label{lim-int}
   \lim_{n\to\infty}\int_a^b\dd[F]\,g_n=0.
\end{equation}
\end{theorem}
\begin{proof}
If $\var_a^bF=0,$ then $\dis\int_a^b\dd[F]\,g_n=\int_a^b\dd[F]\,g=0$ for all $n\inn\N.$
Therefore, without loss of generality we may assume that $\var_a^bF\ne 0.$ Let $\eps>0$
be given. For any $n\in\N$ define
\[
   A_n=\{t\in\ab{:}\,\|g_m(t)\|_X\ge\frac{\eps}{6\,\var_a^bF}
       \mbox{\ for some \ } m\ge n\,\}.
\]
Clearly, $A_{n+1}\subseteq A_n$ for $n\in\N$ and $\bigcap_n A_n=\emptyset.$ Let $v_n$
be given as in \eqref{alpha} with $F$ in the place of $f.$ Recall that $v_n=0$ if
$A_n=\emptyset.$ By Lemma \ref{measure} we have  $\dis\lim_{n\to\infty}v_n=0.$ In particular,
there exists $n_\eps\in\N$ such that
\begin{equation}\label{r6.2}
   v_n<\frac{\eps}{6\,K}\quad \mbox{for \ } n\ge n_\eps.
\end{equation}
As a consequence, we have
\begin{equation}\label{r6.3}
   \var(F,E)<\frac{\eps}{6\,K}
   \quad\mbox{for any \ } n\ge n_\eps
        \mbox{ \ and any \ } E\in\mathcal{E}(A_n).
\end{equation}

\smallskip

Now, fix $n\ge n_\eps.$ Since $g_n$ is regulated, by \cite[Theorem I.3.1]{H} there exists
a step function $h_n{:}\,\ab\,{\to}\,X$ such that
\begin{equation}\label{step}
   \|h_n-g_n\|_\infty<\eps\,',
\end{equation}
where $\eps\,'=\min\left\{\dfrac{\eps}{6\,\var_a^bF},K\right\}.$ In particular, we have
\begin{equation}\label{r6.5}
   \|h_n\|_\infty<K+\eps\,'\leq 2\,K.
\end{equation}
Let us consider the sets
\[
   U_n=\Big\{\,t\in\ab{:}\,\|h_n(t)\|_X\ge\frac{\eps}{3\,\var_a^bF}\,\Big\}
   \quad\mbox{and}\quad
   V_n=\ab{\setminus}\, U_n.
\]
Obviously, $U_n$ and $V_n$ are elementary sets (possibly empty). Furthermore, by \eqref{step}
we obtain
\[
    \frac{\eps}{6\,\var_a^bF}
   =\frac{\eps}{3\,\var_a^bF}-\frac{\eps}{6\,\var_a^bF}
   \le\|h_n(t)\|_X-\eps{\,'}<\|g_n(t)\|_X
   \quad\mbox{for \ } t\in U_n.
\]
This means that $U_n$ is an elementary subset of $A_n,$ which implies that
\begin{equation}\label{r6.6}
    \var(F,U_n)<\frac{\eps}{6\,K}\,.
\end{equation}
Moreover, by \cite[Proposition 10]{Sch1} and \eqref{step} we deduce that
\begin{align*}
   \Big\|\int_a^b \dd[F]\,g_n\Big\|_X
   &\le\Big\|\int_a^b\dd[F]\,(g_n-h_n)\Big\|_X+\Big\|\int_a^b\dd[F]\,h_n\Big\|_X
  \\
   &\le\var_a^b F\,\|g_n-h_n\|_\infty+\Big\|\int_a^b\dd[F]\,h_n\Big\|_X
      <\frac{\eps}{6}+\Big\|\int_a^b\dd[F]\,h_n\Big\|_X\,,
\end{align*}
i.e.
\begin{equation}\label{r6.7}
  \Big\|\int_a^b\dd[F]\,g_n\Big\|_X
  \le\frac{\eps}{6}+\Big\|\int_a^b\dd[F]\,h_n\Big\|_X\,.
\end{equation}
On the other hand, it follows from Theorem \ref{add}, Corollary \ref{estimate-elementary}
and relations \eqref{r6.5}, \eqref{r6.6}, and from the definition of $V_n$ that
\begin{align*}
    \Big\|\int_a^b\dd[F]\,h_n\Big\|_X
   &\le\Big\|\int_{U_n}\dd[F]\,h_n\Big\|_X+\Big\|\int_{V_n}\dd[F]\,h_n\Big\|_X
  \\
   &\le\var(F,U_n)\,\|h_n\|_\infty+\var(F,V_n)\,\Big(\sup_{t\in V_n}\|h_n(t)\|_X\Big)
  \\
   &\le\frac{\eps}{6\,K}\,2\,K+\var_a^bF\,\frac{\eps}{3\,\var_a^bF}
    <\frac{2\,\eps}{3}\,.
\end{align*}
Inserting this into \eqref{r6.7}, we conclude that
\[
   \dis\Big\|\int_a^b\dd[F]\,g_n\Big\|_X{<}\,\eps
   \quad\mbox{for \ } n\ge n_\eps,
\]
wherefrom the desired relation \eqref{lim-int} follows.
\end{proof}

\smallskip

The following assertion will be useful in the proof of Theorem \ref{Osgood} and, to our
knowledge, is not available in literature.

\smallskip

\begin{lemma}\label{break-lim}
Let $F\in BV(\ab, L(X))$ and let $F^B$ be the break function from the Jordan
decomposition of $F.$ Then, for every $g\in G(\ab,X),$ we have
\[
   \int_a^b\dd[F^B]\,g=\sum_{t\in W}\Delta F(t)\,g(t)
\]
where $W$ stands for the set of points of discontinuity of $F$ in $\ab.$
\end{lemma}
\begin{proof}
Since $F$ is regulated, we know that $W$ is a countable set (cf. \cite[Corollary I.3.2]{H}).
Moreover, the sum $\dis\sum_{t\in W}\Delta F(t)\,g(t)$ is well-defined due
to \cite[Lemma 4.1]{MT1}.

Let $W=\{s_k\}$ and note that the function $F^B$ is defined as in \eqref{f-2} (with $F$
in place of $f$). For each $n\in\N,$ consider the function
\[
   F_n^B(t)=\sum_{k=1}^n\big[\Delta^+F(s_k)\,\chi_{(s_k,b]}(t)
                            +\Delta^-F(s_k)\,\chi_{[s_k,b]}(t)\big]
   \quad\mbox{for \ } t\in\ab.
\]
According to Remark \ref{jordan-rem}, we have $\dis\lim_{n\to\infty}\var_a^b(F_n^B-F^B)=0.$
This, together with \cite[Proposition 10]{Sch1}, implies that
\[
    \lim_{n\to\infty}\left\|\int_a^b\dd[F_n^B-F^B]\,g\right\|_X=0,
\]
that is,
\[
    \int_a^b\dd[F^B]\,g=\lim_{n\to\infty}\int_a^b\dd[F_n^B]\,g.
\]
To complete the proof it is enough to determine explicitly the integrals
$\dis\int_a^b\dd[F_n^B]\,g,$ $n\in\N.$  Applying \cite[Proposition 2.3.3]{T1}
(with an obvious extension to Banach space-valued functions), we obtain
\[
   \int_a^b\dd[F_n^B]\,g
   =\sum_{k=1}^n\big[\Delta^+F(s_k)\,g(s_k)+\Delta^-F(s_k)\,g(s_k)\big]
   =\sum_{k=1}^n\Delta F(s_k)\,g(s_k)
\]
wherefrom the result follows.
\end{proof}

\smallskip

\begin{theorem}\label{Osgood}
Let $F\in BV(\ab,X)$ and let a function $g\in G(\ab,X)$ and a sequence $\{g_n\}\subset G(\ab,X)$
be such that
\[
   \lim_{n\to\infty}g_n(t)=g(t)\quad \mbox{for \ }t\in\ab
\]
and
\[
   \|g_n\|_\infty\le K<\infty\quad\mbox{for \ }n\in\N.
\]
Then the integrals $\dis\int_a^b\dd[F]\,g$ and $\dis\int_a^b\dd[F]\,g_n$ exist for all
$n\in\N$ and
\begin{equation}\label{convergence}
   \lim_{n\to\infty}\int_a^b\dd[F]\,g_n=\int_a^b\dd[F]\,g.
\end{equation}
\end{theorem}
\begin{proof} Let $F=F^{\,C}+F^B$ be the decomposition of $F$ as in Theorem \ref{jordan},
that is, $F^{\,C}$ is continuous on $\ab$ and $F^B$ is the break function given by
\[
   F^B(t)=\sum_{a\leq s_k<t}\Delta^+F(s_k)+\sum_{a<s_k\leq t}\Delta^-F(s_k)
   \quad\mbox{for \ } t\in\ab,
\]
where $W=\{s_k\}\subset\ab$ is the set of discontinuities of $F$ in $\ab$.

By \cite[Propostion 15]{Sch1}, we know that $\dis\int_a^b\dd[F]\,g_n$ exists for
each $n\in\N,$ and from the linearity of the integral it follows that
\[
   \int_a^b\dd[F]\,g_n=\int_a^b\dd[F^{\,C}]\,g_n+\int_a^b\dd[F^B]\,g_n.
\]
Put $h_n=g_n-g$ for $n\in\N.$ Therefore $\dis\lim_{n\to\infty}h_n(t)\,{=}\,0,$ and
$\|h_n\|_\infty\,{\le}\,K\,{+}\,\|g\|_\infty$ for every $n\in\N.$ By Lemma \ref{Osgood-cont},
we obtain
\[
   \lim_{n\to\infty}\int_a^b\dd[F^{\,C}]\,h_n=0,
\]
that is,
\[
   \lim_{n\to\infty}\int_a^b\dd[F^{\,C}]\,g_n=\int_a^b\dd[F^{\,C}]\,g.
\]
Now, it is enough to prove the convergence of the sequence $\int_a^b\dd[F^B]\,g_n.$
To this aim, note that by Lemma \ref{break-lim} we have
\[
  \int_a^b\dd[F^B]\,g_n=\sum_{t\in W}\Delta F(t)\,g_n(t)
  \mbox{ \ and \ }
  \int_a^b\dd[F^B]\,g=\sum_{t\in W}\Delta F(t)\,g(t).
\]

If $W$ is finite, it is clear that
\[
  \lim_{n\to\infty}\sum_{t\in W}\Delta F(t)\,g_n(t)=\sum_{t\in W}\Delta F(t)\,g(t).
\]
and consequently \eqref{convergence} holds.

In case $W$ is countable, noting that
$\sum_{k{=}1}^\infty\|\Delta F(s_k)\|_X\le\var_a^bF<\infty$ (see \cite[Lemma 4.1]{MT2}),
there exists $p_\eps\in\N$ such that
\[
  \sum_{k=p_\eps+1}^\infty\|\Delta F(s_k)\|_X<\frac{\eps}{2\,(K+\|g\|_\infty)},
\]
and thus
\[
  \Big\|\sum_{k=p_\eps+1}^\infty\Delta F(s_k)\,[g_n(s_k)-g(s_k)]\Big\|_X
  \le(K+\|g\|_\infty)\sum_{k=p_\eps+1}^\infty\|\Delta F(s_k)\|_X
  <\frac{\eps}{2}.
\]
On the other hand, we may choose $n_\eps\in\N$ so that
\[
   \Big\|\sum_{k=1}^{p_\eps}\Delta F(s_k)\,[g_n(s_k)-g(s_k)]\Big\|_X<\frac{\eps}{2}
   \quad\mbox{for \ }n\ge n_\eps.
\]
With this in mind, we obtain
\[
    \Big\|\int_a^b\dd[F^B]\,(g_n-g)\Big\|_X
   =\Big\|\sum_{k=1}^\infty\Delta F(s_k)\,[g_n(s_k)-g(s_k)]\Big\|_X<\eps,
	\quad\mbox{ for \ } n\ge n_\eps
\]
which completes the proof.
\end{proof}



\begin{thebibliography}{99}

\bibitem{Apo}
T.~M.~Apostol,
{\sl Mathematical Analysis} (2nd ed.).
Addison-Wesley, Reading, MA, 1974.

\bibitem{BaSh}
R.~G.~Bartle and D.~R.~Sherbert,
{\sl Introduction to Real Analysis.}
Wiley, New York, 1982.

\bibitem{Ch}
V.~V.~Chistyakov,
On the theory of set-valued maps of bounded variation of one real variable.
{\sl  Sb. Math.} {\bf 189} (1998), 153--176.

\bibitem{Di}
J.~Diestel,
{\sl Sequences and Series in Banach spaces.}
Springer-Verlag, New York-Berlin, 1984.

\bibitem{E}
W.~F.~Eberlein,
Notes on integration I: the underlying convergence theorem.
{\sl Comm. Pure Appl. Math.} {\bf 10} (1957) 357--360.

\bibitem{G}
R.~A.~Gordon,
{\sl The integrals of Lebesgue, Denjoy, Perron, and Henstock.}
Graduate Studies in Math., AMS, Providence, Rhode Island, 1994.

\bibitem{Hi}
T.~H.~Hildebrandt,
{\sl Theory of Integration.} Academic Press, New York \& London, 1963.

\bibitem{H1}
C.~S.~H\"onig,
{\sl The abstract Riemann-Stieltjes integral and its applications to linear
differential equations with generalized boundary conditions.}
Notas do Instituto de Matem\'atica e Estat\'istica da Universidade de S. Paulo,
S\'erie de Matem\'atica no. 1, 1973.

\bibitem{H}
C.~S.~H\"onig,
{\sl Volterra Stieltjes-integral equations.}
North Holland and American Elsevier, Mathematics Studies 16. Amsterdam and New York, 1975.

\bibitem{Lewin}
J.~W.~Lewin,
A Truly Elementary Approach to the Bounded Convergence Theorem.
{\sl  Amer. Math. Monthly} {\bf 93} (1986), 395--397.

\bibitem{Lu}
W.~A.~J.~Luxemburg,
Arzela's dominated convergence theorem for the Riemann integral.
{\sl  Amer. Math. Monthly} {\bf 78} (1971) 970--979.

\bibitem{MT1}
G.~A.~Monteiro and M.~Tvrd\'y,
On Kurzweil-Stieltjes integral in Banach space.
{\sl Math. Bohem.} {\bf 137} (2012), 365--381.

\bibitem{MT2}
G.~A.~Monteiro and M.~Tvrd\'y,
Generalized linear differential equations in a~Banach space: Continuous dependence
on a parameter.
{\sl Discrete Contin. Dyn. Syst.} {\bf 33} (1) (2013), 283--303,
doi:10.3934/dcds.2013.33.283.

\bibitem{Ru1}
W.~Rudin,
{\sl Principles of Mathematical Analysis.}
McGraw-Hill, New York, 1964.

\bibitem{Ru2}
W.~Rudin,
{\sl Real and Complex Analysis.}
McGraw-Hill, New York, 1966.

\bibitem {S}
S.~Saks,
{\sl Theory of the Integral.}
Second Edition, Hafner Publishing Company, New York, 1937.

\bibitem{Sch}
\v{S}.~Schwabik,
{\sl Generalized Ordinary Differential Equations, World Scientific.}
Singapore, 1992. Zbl 0781.34003

\bibitem{Sch1}
\v{S}.~Schwabik,
Abstract Perron-Stieltjes integral. {\sl Math. Bohem.} {\bf 121} (1996), 425--447.

\bibitem{Sch2}
\v{S}.~Schwabik,
Linear Stieltjes integral equations in Banach spaces. Math. Bohem. {\sl 124} (1999) 433--457.
Zbl 0937.34047

\bibitem{Sch3}
\v{S}.~Schwabik,
Linear Stieltjes integral equations in Banach spaces II; Operator valued solutions.
Math. Bohem. {\sl 125} (2000), 431--454.
Zbl 0974.34057

\bibitem{STV}
\v{S}.~Schwabik, M.~Tvrd\'y, O.~Vejvoda,
{\sl Differential and Integral Equations: Boundary Value Problems and Adjoints.}
Academia and  D. Reidel, Praha and Dordrecht, 1979.

\bibitem{T1}
M.~Tvrd\'y,
Differential and integral equations in the space of regulated functions,
{\sl Mem. Differential Equations Math. Phys.} 25 (2002), 1-104.
\end{thebibliography}
\end{document}